\documentclass[12pt]{amsart}

\usepackage{amsmath, amssymb, graphics, caption}

\newcommand{\mathsym}[1]{{}}
\newcommand{\unicode}[1]{{}}

\newtheorem{thm}{Theorem}[section]
\newtheorem{cor}[thm]{Corollary}
\newtheorem{lem}[thm]{Lemma}
\newtheorem{prop}[thm]{Proposition}
\theoremstyle{definition}
\newtheorem{notation}[thm]{Notation}
\newtheorem{defn}[thm]{Definition}
\newtheorem{rem}[thm]{Remark}

\newtheorem*{defn*}{Definition}
\newtheorem*{rems*}{Remarks}
\newtheorem*{rem*}{Remark}
\newtheorem{ex}[thm]{Example}

\newtheorem*{Def}{Definition}
\numberwithin{equation}{section}
\begin{document}

\title [Singularities of affine equidistants] {Singularities of affine equidistants:\\ extrinsic geometry of surfaces in 4-space.}

\author[W. Domitrz]{W. Domitrz}
\address{Warsaw University of Technology, Faculty of Mathematics and Information Science, Plac Politechniki 1, 00-661 Warszawa, Poland}
\email{domitrz@mini.pw.edu.pl}
\author[S. Janeczko]{S. Janeczko}
\address{Warsaw University of Technology, Faculty of Mathematics and Information Science, Plac Politechniki 1, 00-661 Warszawa, Poland}
\email{janeczko@mini.pw.edu.pl}
\author[P. de M. Rios]{P. de M. Rios}
\address{Departamento de Matem\'atica, ICMC, Universidade de S\~ao Paulo; S\~ao Carlos, SP, 13560-970, Brazil}
\email{prios@icmc.usp.br}
\author[M. A. S. Ruas]{M. A. S. Ruas}
\address{Departamento de Matem\'atica, ICMC, Universidade de S\~ao Paulo; S\~ao Carlos, SP, 13560-970, Brazil}
\email{maasruas@icmc.usp.br}

\thanks{W. Domitrz and S. Janeczko were partially supported by NCN grant no. DEC-2013/11/B/ST1/03080. P. de M. Rios was partially supported by FAPESP grant no. 2013/04630-9. M. A. S. Ruas was partially supported by FAPESP grant no. 2014/00304-2 and CNPq grant no. 305651/2011-0.}

%\subjclass{57R45, 58K40, 53D12, 58K25, 58K50.}

%\keywords{Symplectic geometry, Lagrangian singularities}

\maketitle

\begin{abstract}
%We study $\lambda-$equidistants of a smooth closed surface $M$ in $\mathbb R^4.$

For a generic embedding of a smooth closed surface $M$ into
$\mathbb R^4,$ the subset of $\mathbb R^4$ which is the affine $\lambda-$equidistant of $M$   appears as the discriminant set of a stable mapping $M \times M \to \mathbb R^4,$ hence their stable singularities are $A_k, \, k=2, 3, 4,$ and $C_{2,2}^{\pm}.$ In this paper, we  characterize these stable singularities of $\lambda-$equidistants in terms of the bi-local extrinsic geometry of the surface, leading to a geometrical study of the set of weakly parallel points on $M$.

\end{abstract}

\section{Introduction}

When $M$ is a smooth closed curve on the affine plane $\mathbb R^2$, the set of all midpoints of chords connecting pairs of points on $M$ with parallel tangent vectors is called the {\it Wigner caustic} of $M$, or the {\it area evolute} of $M$, or still, the {\it affine $1/2$-equidistant} of $M$.
The
$1/2$-equidistant is generalized to any $\lambda$-equidistant, denoted $E_{\lambda}(M)$, $\lambda\in\mathbb R$, by considering all chords  connecting pairs of points of $M$ with parallel tangent vectors and the set of all points of these chords  which stand in the $\lambda$-proportion to their corresponding pair of points on $M$.

The definition of the affine $\lambda$-equidistant of $M$ is generalized to the cases when $M$ is an $n$-dimensional closed submanifold of $\mathbb R^q$, with
$q\leq 2n$, by considering the set of all $\lambda$-points of chords connecting pairs of points on $M$ whose direct sum of tangent spaces do not coincide with $\mathbb R^q$, the so-called {\it weakly parallel pairs} on $M$. In the particular case  of  $M^2\subset\mathbb R^4$, a weakly parallel pair on the surface $M$ can be either $1$-parallel (when the tangent spaces span a $3$-space) or $2$-parallel, which is the case of true parallelism, also called strong parallelism.

Affine equidistants of smooth submanifolds, in particular the Wigner caustic, have a way in mathematical physics and in the definition of affine-invariant global centre symmetry sets of these submanifolds and, in every case, precise knowledge of their singularities is an important issue \cite{Jan, GZ1, GZ2, DRs, DMR}.
Thus, stable singularities of affine equidistants of $M^n\subset\mathbb R^q$ have been extensively studied  \cite{Ber, GH, GZ1, GZ2, GJ, DRs}, culminating in its complete classification for all pairs $(2n,q)$ of nice dimensions \cite{DRR}.

On the other hand, not so much is known with respect to the interpretation for the realization of these stable singularities in terms of the extrinsic geometry of $M^n\subset\mathbb R^q$. The case of curves on the plane has long been well understood \cite{Ber, GH}, just as for hypersurfaces \cite{GZ1}. Another instance that has been completely worked out refers to  a Lagrangian surface $M^2$ in symplectic $\mathbb R^4$, for its  Wigner caustic on shell, that is, the part of its $1/2$-equidistant that is close to and contains $M$ \cite{DMR}. A geometric study of the Wigner caustic on shell for general surfaces in $\mathbb R^4$ has also been partly worked out in \cite{GJ}.

In this paper, we extend the extrinsic geometric study of the realization of affine equidistants to the case of general (off-shell) $\lambda$-equidistants of any smooth surface $M^2\subset\mathbb R^4$.  Our paper is organized as follows:

First, Section 2 reviews basic definitions and characterizations of affine equidistants. The presentation is based on \cite{DRR}. Then, basic facts on extrinsic geometry of surfaces in $4$-space are recalled in Section 3.

Our geometric study is presented in Sections 4 and 5. First, in Section 4 we describe the realization of singularities of affine equidistants in terms of the bi-local extrinsic geometry of the surface. The main result for the case of $1$-parallel pairs is presented in Theorem  \ref{prop-A_k}, while Theorems \ref{c+} and \ref{c-} present the results for the $2$-parallel case.

Then, this is followed in Section 5 by a complementary study of the set of weakly parallel points on $M$. We start by using the Grassmannian of $2$-planes in $4$-space, cf. Propositions \ref{singW} and \ref{Wpi2} and Theorem \ref{thmgrass}, leading to the final detailed description of the set of weakly parallel points on $M$ presented in Corollary \ref{corgrass} and Theorem \ref{weaklyset}.

\

\noindent{\it Acknowledgements}: This work started while the last two authors (P.R. \& M.R.) were visiting Warsaw and continued during visits of the first two authors (W.D \& S.J.) to S\~ao Carlos. We all thank the respective host institutes for hospitality and the funding agencies (NCN, FAPESP \& CNPq) for financial support for these visits.

\section{Singularities of affine equidistants: overview}\label{revsing}

In this section, we summarize the material that is presented in \cite{DRR} in greater detail, in order to describe, characterize and classify the singularities of affine $\lambda$-equidistants of smooth submanifolds.

\subsection{Definition of affine equidistants}

Let $M$ be a smooth closed $n$-dimensional submanifold of
the affine space $\mathbb R^{q}$, with $q\leq 2n$. Let $\alpha, \beta$ be points of $M$ and denote by
$\tau_{\alpha-\beta}:\mathbb R^q \ni x\mapsto x+(\alpha-\beta) \in \mathbb R^q$ the
translation by the vector $(\alpha-\beta)$.
\begin{defn}\label{parallelism} A pair of points $\alpha, \beta \in M$ ($\alpha\ne \beta$) is called a
{\bf weakly parallel} pair if
$$T_{\alpha}M + \tau_{\alpha-\beta}(T_{\beta}M)\ne \mathbb R^q.$$

A weakly parallel pair $\alpha, \beta \in M$ is
called {\bf $k$-parallel} if
$$\dim(T_{\alpha}M \cap \tau_{\alpha-\beta}(T_{\beta}M))=k.$$
If $k=n$ the pair $\alpha, \beta \in M$ is called {\bf strongly parallel},
or just {\bf parallel}. We also refer to $k$ as the {\bf degree of
parallelism} of the pair $(\alpha,\beta)$.
\end{defn}

\begin{defn}
A {\bf chord} passing through a pair $\alpha,\beta$, is the line
$$
l(\alpha,\beta)=\{x\in \mathbb R^q|x=\lambda \alpha + (1-\lambda) \beta, \lambda \in
\mathbb R\},
$$
but we sometimes also refer to $l(\alpha,\beta)$ as a chord {\it joining} $\alpha$ and $\beta$.
\end{defn}

\begin{defn}\label{defElambda} For a given $\lambda$, an {\bf affine
$\lambda$-equidistant} of $M$, $E_{\lambda}(M)$, is the set of all
$x\in \mathbb R^q$ such that $x=\lambda \alpha + (1-\lambda) \beta$, for
all weakly parallel pairs $(\alpha,\beta)$ in $M$. $E_{\lambda}(M)$ is also
called a {\bf momentary equidistant} of $M$.
Whenever $M$ is understood, we write
$E_{\lambda}$ for $E_{\lambda}(M)$.
\end{defn}
Note that, for any
$\lambda$, $E_{\lambda}(M)=E_{1-\lambda}(M)$ and in particular
$E_0(M)=E_1(M)=M$. Thus, the case $\lambda=1/2$ is special:
\begin{defn}$E_{1/2}(M)$ is called the {\bf Wigner caustic} of $M$  \cite{Ber, OH}.
\end{defn}

\subsection{Characterization of affine equidistants by projection}

Consider the product affine space: $\mathbb R^{q}\times \mathbb R^{q}$ with coordinates $(x_+,x_-)$
and the tangent bundle to $\mathbb R^{q}$:  $T\mathbb
R^{q}=\mathbb R^{q}\times \mathbb R^{q}$ with coordinate system
$(x,\dot{x})$ and standard projection
$\pi: T\mathbb R^{q}\ni (x,\dot{x})\rightarrow x \in \mathbb R^{q}$.

\begin{defn} $\forall \lambda\in\mathbb{R}\setminus  \{0,1\}$, a {\bf $\lambda$-chord transformation}
$$\Gamma_{\lambda}:\mathbb{R}^{q}\times\mathbb{R}^{q}\to T\mathbb{R}^{q} \ , \ (x^+,x^-)\mapsto(x,\dot{x})$$
is a  linear diffeomorphism  defined by:
\begin{equation}\label{x}
x=\lambda x^+ + (1-\lambda)x^- \ , \ \
\dot{x}=x^+ - x^-.
\end{equation}
\end{defn}

\begin{rem} The choice of linear equation for $\dot{x}$ in (\ref{x}) is not unique, but this is the simplest one. Among other possibilities, the choice $\dot{x}=\lambda x^+ - (1-\lambda)x^-$ is particularly well suited for the study of affine equidistants of {\it Lagrangian} submanifolds in symplectic space \cite{DRs}.
\end{rem}

Now, let $M$ be a smooth closed $n$-dimensional submanifold of the
affine space $\mathbb R^{q}$ ($2n\ge q$) and consider the product
$M\times M\subset \mathbb R^{q}\times\mathbb R^{q}$. Let $\mathcal
M_{\lambda}$ denote the image of $M\times M$ by
a  $\lambda$-chord transformation,
$$\mathcal M_{\lambda} = \Gamma_{\lambda}(M\times M) \ ,$$
which is a $2n$-dimensional smooth submanifold of $T\mathbb
R^{q}$.

Then we have the following general characterization:
\begin{thm}[\cite{DRs}]\label{gensing}
The set of critical values of the standard projection $\pi:
T\mathbb R^{q}\to\mathbb R^{q}$ restricted to $\mathcal
M_{\lambda}$ is $E_{\lambda}(M)$.
\end{thm}

\begin{defn}\label{pilambdamap} $\forall \lambda\in\mathbb{R}\setminus  \{0,1\}$, the {\bf $\lambda$-point map} is the projection
$$\Psi_{\lambda}:  \mathbb R^{q}\times\mathbb R^{q}\to\mathbb R^{q} \ , \ (x^+,x^-)\to x=\lambda x^+ + (1-\lambda)x^- \ .$$
\end{defn}

\begin{rem}\label{psilambda}
Because $\Psi_{\lambda}=\pi\circ\Gamma_{\lambda}$ we can rephrase Theorem \ref{gensing}: {\it the  set of critical values of the projection $\Psi_{\lambda}$ restricted to $M\times M$ is $E_{\lambda}(M)$}.
\end{rem}

\subsection{Characterization of affine equidistants by contact}

In the literature, if  $M\subset\mathbb R^{2}$ is a smooth curve, the Wigner caustic $E_{1/2}(M)$
has been described in various ways, one of which  says that,
if ${\mathcal R}_a:\mathbb R^{2}\to\mathbb R^{2}$
denotes reflection through $a\in\mathbb R^{2}$, then $a\in E_{1/2}(M)$ when $M$ and  ${\mathcal R}_a(M)$ are not
transversal \cite{Ber, OH}. We generalize this description for every $\lambda$-equidistant of submanifolds of more arbitrary dimensions.

\begin{defn}\label{reflection}
$\forall  \lambda\in\mathbb{R}\setminus  \{0,1\}$, a  $\lambda$-{\bf reflection} through $a\in  \mathbb R^{q}$ is the map
\begin{equation}\label{reflect}
{\mathcal R}_a^{\lambda}: \mathbb R^{q}\to\mathbb R^{q} \ , \ x\mapsto {\mathcal R}_a^{\lambda}(x)=\frac{1}{\lambda}a-\frac{1-\lambda}{\lambda}x
\end{equation}
\end{defn}
\begin{rem} A  $\lambda$-reflection through $a$ is not a reflection in the strict sense because ${\mathcal R}_a^{\lambda}\circ{\mathcal R}_a^{\lambda}\neq id: \mathbb R^{q}\to\mathbb R^{q}$, instead,
$${\mathcal R}_a^{1-\lambda}\circ{\mathcal R}_a^{\lambda}= id: \mathbb R^{q}\to\mathbb R^{q} \ ,$$
so that, if $a=a_{\lambda}=\lambda a^+ + (1-\lambda)a^-$ is the $\lambda$-point of $(a^+,a^-)\in \mathbb R^{2q}$,
$${\mathcal R}_{a_{\lambda}}^{\lambda}(a^-)=a^+ \ , \  {\mathcal R}_{a_{\lambda}}^{1-\lambda}(a^+)=a^- \ .$$
Of course, for $\lambda=1/2$,  ${\mathcal R}_a^{1/2} \equiv {\mathcal R}_a$ is a reflection in the strict sense.
\end{rem}

Now, let $M$ be a smooth $n$-di\-men\-sion\-al submanifold of $\mathbb R^{q}$, with $2n\geq q$. Also, let $M^+$ be a germ of submanifold $M$ around $a^+$, let $M^-$ be a germ of submanifold $M$ around $a^-$ and let $a=a_{\lambda}=\lambda a^+ + (1-\lambda)a^-$ be the $\lambda$-point of $(a^+,a^-)\in M\times M\subset \mathbb R^{q}\times\mathbb R^{q}$.

Then, the following characterization is immediate:
\begin{prop} The following conditions are equivalent:
\begin{itemize}
\item[(i)] $a\in E_{\lambda}(M)$

 \item [(ii)] $M^+$ and ${\mathcal R}_{a}^{\lambda}(M^-)$ are not transversal at $a^+$

 \item[(iii)]  $M^-$ and ${\mathcal R}_{a}^{1-\lambda}(M^+)$ are not transversal at $a^-$.

 \end{itemize}
\end{prop}

Therefore, {\it the study of the singularities of $E_{\lambda}(M)\ni 0$ can be proceeded via the study of the contact between $M^+$ and ${\mathcal R}_0^{\lambda}(M^-)$ or, equivalently, the contact between  ${\mathcal R}_0^{1-\lambda}(M^+)$ and $M^-$}.

\subsection{Singularities of contact}

Let $N_1, N_2$ be germs at $x$ of smooth $n$-dimensional submanifolds of
the space $\mathbb R^{q}$, with $2n\ge q$. We describe $N_1,N_2$  in the following way:
\begin{itemize}
\item
$N_1=f^{-1}(0)$, where $f:(\mathbb R^q,x)\rightarrow (\mathbb R^{q-n},0)$ is a submersion-germ,
\item
$N_2=g(\mathbb R^n)$, where $g:(\mathbb R^n,0)\rightarrow (\mathbb R^q,x)$ is an embedding-germ.
\end{itemize}

\begin{defn}
A {\bf contact map} between submanifold-germs $N_1, N_2$ is the following map-germ
$\kappa_{N_1,N_2} :(\mathbb R^n,0)\rightarrow (\mathbb R^{q-n},0),$ where $\kappa_{N_1, N_2}=f\circ g.$
\end{defn}

Let $\tilde{N_1}, \tilde{N_2}$ be another pair of germs at $\tilde{x}$ of smooth $n$-dimensional submanifolds of
the space $\mathbb R^{q}$, described in the same way as $N_1,N_2$.

\begin{defn}
The contact of $N_1$ and $N_2$ at $x$ is of the same {\bf contact-type} as the contact of $\tilde{N_1}$ and $\tilde{N_2}$ at $\tilde{x}$ if $\exists$ a diffeomorphism-germ
$\Phi:(\mathbb R^q,x)\rightarrow (\mathbb R^q,\tilde{x})$ s.t. $\Phi(N_1)=\tilde{N_1}$ and $\Phi(N_2)=\tilde{N_2}$. We denote the contact-type of $N_1$ and $N_2$ at $x$ by $\mathcal K(N_1,N_2,x)$.
\end{defn}

\begin{thm}[\cite{Mont}]
$\mathcal K(N_1,N_2,x) =\mathcal K(\tilde{N_1},\tilde{N_2},\tilde{x})$ if and only if the contact maps $f\circ g$ and $\tilde{f}\circ \tilde{g}$ are $\mathcal K$-equivalent.
\end{thm}

\begin{defn}
We say that $N_1$ and $N_2$ are $k$-{\bf tangent} at $x=0$ if
$$\dim (T_0N_1 \cap T_0N_2)=k \ .$$
If $k$ is maximal, that is, $ k=\dim(T_0N_1)=\dim(T_0N_2)$,
we say that $N_1$ and $N_2$ are {\bf tangent} at $0$.
\end{defn}

\begin{rem}\label{tan-par} In  the context of affine equidistants, $E_{\lambda}(M)$, note that  $N_1=M^+$ and $N_2=\mathcal R_0^{\lambda}(M^-)$ are $k$-{\bf tangent} at $0$  if and only if $T_{a^+}M^+$ and $T_{a^-}M^-$ are $k$-{\bf parallel}, where $\lambda a^+ +(1-\lambda) a^- = 0\in E_{\lambda}(M)$. \end{rem}

\begin{prop}[\cite{DRR}]
If $N_1$ and $N_2$ are $k$-tangent at $0$ then the corank of the contact map $\kappa_{N_1,N_2}$ is $k$.
\end{prop}

\section{Extrinsic geometry of surfaces in 4-space: overview}

%The geometric interpretation for the realizations of singularities of affine equidistants of curves in   $\mathbb R^2$ and surfaces in $\mathbb R^3$ has been done in \cite{  }.
%In this paper, we will present  the interpretation for the realizations of singularities of affine equidistants of surfaces in  $\mathbb R^4$ in terms of their extrinsic geometry. For this reason, before proceeding with our study it is useful to

In this section, we remind  basic definitions and results on the extrinsic geometry of smooth surfaces in $4$-space. See \cite{li, MRR}  for details.

Let  ${f}: U \to \mathbb R^4$ be a local parametrisation of $M$, where $U$ is an open subset of $\mathbb R^2$. 
Let $\{{\bf e}_1, {\bf e}_2, {\bf e}_3, {\bf e}_4\}$ be a positively oriented orthonormal  frame in $\mathbb R^4$  such that at any $y=(y_1,y_2) \in U,$ $\{ {\bf e}_1(y), {\bf e}_2(y)\} $ is a basis for the tangent plane $T_pM$ 
and $\{{\bf e}_3(y), {\bf e}_3(y)\}$  is a basis for the normal plane $N_pM$ at $p={{f}(y)}$.

\begin{defn} The {\bf second fundamental form} of $M$ at $p$ is the vector valued quadratic form
${\rm II}_p: T_pM \to N_pM$ associated 
to the normal component of the second derivative $d^2{f}$ of ${f}$ at $p,$ that is, 
$$
{\rm II}_p=\langle d^2{f}, {\bf e}_3\rangle  {\bf e}_3+ \langle d^2{f}, {\bf e}_4\rangle  {\bf e}_4.
$$ 
\end{defn}

Let $a=\left\langle{\bf e_3}, {f}_{y_1y_1}\right\rangle,  b= \left\langle{\bf e_3}, {f}_{y_1y_2}\right\rangle,  c=\left\langle{\bf e_3}, {f}_{y_2y_2}\right\rangle, $
 $e=\left\langle{\bf e_4}, {f}_{y_1y_1}\right\rangle,  f= \left\langle{\bf e_4}, {f}_{y_1y_2}\right\rangle,  g=\left\langle{\bf e_4}, {f}_{y_2y_2}\right\rangle.$
 
 Then, with this notation,  we can write $${\rm II_p}({\bf u})=(au_1^2+2b u_1u_2+cu_2^2){\bf e_3}+ (eu_1^2+2fu_1u_2+gu_2^2){\bf e_4},$$ where ${\bf u}=u_1{{\bf e_1}}+u_2 { {\bf e_2}} \in T_pM.$

\vspace{0.3cm} 
 
 The matrix $
\alpha=\left (\begin{array}{ccc}
a& b& c\\
e & f& g
\end{array} \right)
$ is called the matrix of the second fundamental form with respect to the orthonormal frame  $\{{\bf e}_1, {\bf e}_2, {\bf e}_3, {\bf e}_4\}.$

\begin{defn} 
	The second fundamental form of $M$ at $p,$ along a normal vector field $\nu$  is the quadratic form
	$\rm II^{\nu}_{p}: T_p M \to \mathbb R$ defined by
	$${\rm II}^{\nu}_{p}({\mbox{\bf u}})=\left <{\rm II_p}({\bf \mbox{\bf u}}), {\mbox{\bf v}} \right >, \,\,\,\, {\bf \mbox{\bf u} \in T_pM},\, {\mbox{\bf v}}=\nu(p) \in N_pM,$$  where ${\rm II_p}({\bf \mbox{\bf u}}): T_pM \to N_pM$  is the second fundamental form at $p.$ 
\end{defn}

 Let $S^1$ be the unit circle in $T_pM$ parametrized by $\theta \in [0, 2\pi].$  Denote by $\gamma_{\theta}$ the curve obtained by
 intersecting $M$ with the hyperplane at $p$ composed by the direct sum of
 the normal plane $N_{p}M$ and the straight line in the tangent direction
 represented by $\theta$. Such curve is called {\em normal section of $M$ in
 	the direction $\theta$}.
 
 %We denote by $\gamma_{\theta}$ the normal section of $M$ in the direction ${\bf u}_{\theta}$ parametrised by arclength.
 
 \begin{Def}  The {\bf curvature ellipse} is the image of the mapping
 	$$\begin{array}{ccll} \eta : & S^1  & \longrightarrow & N_pM \\
 	& {\theta} & \longmapsto &  \eta ({\theta}),
 	\end{array},$$ where $\eta(\theta)$ is the curvature vector of $\gamma_{\theta}.$
 \end{Def}

Scalar invariants of the extrinsic geometry of surfaces in $\mathbb R^4$ can be defined using the coefficients of the second fundamental form.
For instance the Gaussian curvature
\begin{equation}\label{eq:Curv4}
\mathcal G_M=ac-b^2+eg-f^2
\end{equation} and the  $\Delta$ function
\begin{equation}\label{eq:Delta4}
\Delta_M= \frac{1}{4} det \left [\begin{array}{cccc}
a & 2b & c & 0 \\
e & 2f & g & 0 \\
0 & a & 2b & c \\
0 & e & 2f  & g  \\
\end{array} \right].
\end{equation}
%We can also write
%\begin{equation}
%\Delta_M=\frac{1}{4}\{4(a_1b_2-a_2b_1)(b_1c_2-b_2c_1)-(a_1c_2-a_2c_1)^2\}
%\end{equation}
%We quote now the following facts whose detailed proof can be found in \cite{Li}

Although neither  $\Delta_M$ nor $\mathcal G_M$ are affine invariants (a chosen metric was used to define them), the following proposition  allows for an affine-invariant classification of  a point $p\in M \subset \mathbb R^4$.

\begin{prop}[\cite{DMR}, Proposition 4.18]\label{signal} The sign of $\Delta_M$ is an affine invariant. When $rank\{II_{(p)}\}=1$, the sign of  $\mathcal G_M$  is also an affine  invariant. \end{prop}

%A proof of Proposition \ref{signal} is found in \cite{DMR}. Now, we recall \cite{MRR} that:

\begin{defn} A point $p\in M$ is called

\noindent (i)  {\bf parabolic} if $\Delta_M(p)=0$,

\noindent (ii) {\bf elliptic} if $\Delta_M(p)>0$,

\noindent (iii) {\bf hyperbolic} if $\Delta_M(p)<0$.

\end{defn}

\begin{defn} A parabolic point $p\in M$ is called

\noindent (i-i) {\bf point of nondegenerate ellipse}, if $rank\{II_{(p)}\}=2$.

When $rank\{II_{(p)}\}=1$, $p$ is an inflection point. In this case, it is 

\noindent (i-ii) {\bf inflection point of real type}, if $\mathcal G_M(p)<0$, 

\noindent (i-iii) {\bf inflection point of flat type}, if $\mathcal G_M(p)=0$.

\noindent (i-iv)  {\bf inflection point of imaginary type}, if $\mathcal G_M(p)>0$,
\end{defn}

\begin{defn}
A direction ${\bf v} \in N_{p}M$ is a {\bf binormal direction} at $p$ if the second fundamental form ${\rm II^{{\bf v}}_p}$ along the ${\bf v}$ direction is a degenerate quadratic form. In this case, a direction ${\bf u} \in T_pM$ in the kernel  of ${\rm II^{{\bf v}}_p}({\bf u})$ is called an  {\bf asymptotic direction}.
\end{defn}

%In \cite{MRR}  the different points of the
%surfaces were characterized  in terms of the singularities of the family of height functions on them. In particular,
%it follows from this approach that:
%\begin{itemize}
%\item[(i)]  if $\Delta_M (p) < 0$ then there are 2 binormal directions and 2 associated asymptotic
%directions at $p;$
%\item[(ii)] if  $\Delta_M (p) > 0$ then there are  no  binormal directions and no  asymptotic
%directions at $p;$
%\item [(iii)] if $\Delta_M (p) = 0$ and $p$ is not an inflection point then there is 1 binormal direction and 1 associated %asymptotic direction at $p;$
%\item [(iv)] if $p$ is an inflection point then ${\rm II^{{\bf v}}_p}=0$ and every direction ${\bf u} \in T_pM$ is an asymptotic %direction.
%\end{itemize}

\begin{defn} For a surface  $M\subset\mathbb R^4$, $p\in M$ and ${\bf u}\in T_pM$, ${\bf v}\in N_pM$, we say that $({\bf u}, {\bf v})$  is a {\bf contact pair} of $M$ at $p$ if ${\bf v}$ is a binormal direction at $p$ and ${\bf u}$ is an asymptotic direction associated to ${\bf v}$.
\end{defn}

\begin{prop} [\cite{MRR}, Lemma 3.2] Let $M$ be a surface in $\mathbb R^4$,

\noindent 1) For a hyperbolic point $p\in M$, there are exactly $2$ contact pairs at $p$.

\noindent 2) For an elliptic point $p\in M$, there are no contact pairs at $p$.

\noindent 3) For a parabolic point $p\in M$,

i) if $p$ is a point of nondegenerate ellipse, then there exists only one contact pair at $p$.

ii) if $p$ is an inflection point, then there exists only one  ${\bf v}\in N_pM$ such that, for all ${\bf u}\in T_pM$,
\ $({\bf u},{\bf v})$ is a contact pair at $p$.
\end{prop}

\section{Extrinsic geometry of surfaces in 4-space and singularities of their affine equidistants}

We now present the geometric interpretation for the realizations of stable singularities of affine equidistants of surfaces in  $\mathbb R^4$.

We first recall the following result from  \cite{DRR}:

\begin{thm}[\cite{DRR}, Theorem 5.2]\label{eq:stable}
There exists a residual set  $\mathcal S$ of embeddings $i: M^2 \to \mathbb R^4,$ such that the map $\Psi_{\lambda}: M\times M \setminus \Delta \to \mathbb R^4$  is locally stable, where $\Psi_{\lambda}(x,y)=\lambda i(x)+(1-\lambda)i(y)$ and $\Delta$ is the diagonal in $M\times M$.
\end{thm}

\begin{defn}\label{genericembedding} We say that  $i: M^2 \to \mathbb R^4$ is a {\bf generic embedding} if $i \in \mathcal S.$\end{defn}

Because the codimension of each singularity of $\Psi_{\lambda}$ is at most $4$,  the possible stable singularities of affine equidistants of surfaces in $\mathbb R^4$ are: $$A_1, A_2, A_3, A_4 \  \  \mbox{for 1-parallelism,} \quad C^{\pm}_{2,2} \  \ \mbox{for 2-parallelism.}$$

For the reader's convenience, we recall the normal forms of these stable singularities $(\mathbb R^4,0) \to (\mathbb R^4,0)$ in the table below:

\begin{center}
\begin{table}[h]
    \begin{small}
    \noindent
    \begin{tabular}{|c|c|c|}
            \hline
    Notation &  Normal form  \\ \hline
      $ A_2$ & $(u_1, u_2, u_3, y^2)$\\
     $A_{\mu},  2\leq \mu\leq 4$ & $(u_1,u_2,u_3, y^{\mu+1}+\Sigma_{i=1}^{\mu-1} u_iy^{i})$ \\ \hline
     $C_{2,2}^{+}$ & $(u_1, u_2, x^2+u_1y, y^2+u_2x)$ \\ \hline
     $C_{2,2}^{-}$ & $(u_1, u_2, x^2-y^2, xy+u_1x+u_2y)$  \\ \hline
     %$E_7$ & $y_1^3+y_1y_2^3+Q_{s-2}$ &  -\\ \hline
     %$E_8$ & $y_1^3+y_2^5+Q_{s-2}$ &  -\\ \hline
\end{tabular}
%\smallskip
%\caption{}
%\caption{\small $\mathcal K$-simple germs $\mathbb
%R^s\rightarrow \mathbb R$.   $Q_{s-i}=\pm y_{i+1}^2\pm \cdots \pm
%y_s^2$.}\label{ADE}
\end{small}
\end{table}
\end{center}

We refer to \cite{DRR}, where all possible stable singularities of affine equidistants are classified for submanifolds $M^n\subset \mathbb R^q$, with $(2n,q)$ an arbitrary pair of nice dimensions, for all possible degrees of parallelism.

In this paper, we focus on investigating the conditions for realizing these equidistant singularities $A_{\mu},\, 1\leq \mu \leq 4$ and $C^{\pm}_{2,2}$ from the extrinsic geometry of a  generic embedding  of smooth surface $M\subset \mathbb R^4$.

 In this specific case  we
substitute submanifold-germs $N_1$ and $N_2$ of Section \ref{revsing} by $N_1=M^+$ and $N_2=\mathcal R_0^{\lambda}(M^-)$, or equivalently by $N_1=M^-$ and $N_2=\mathcal R_0^{1-\lambda}(M^+)$, where $M_{+}$ is the surface-germ of $M$ around $a^{+}\in M\subset\mathbb R^4$ and $M_{-}$ is the surface-germ of $M$ around $a^{-}\in M\subset\mathbb R^4$, with $\lambda a^+ + (1-\lambda)a^- =0$.

\subsection{Bi-local geometry of weakly parallel pairs and singularities of affine equidistants}

We start by looking at the bi-local  geometry of $1$-parallel pairs.

Suppose $(a^+, a^-)$ is a pair of 1-parallel points. Then, we can choose coordinates in a neighbourhood of  $a^+$ and $a^-$ as follows:
\begin{equation}
\begin{split}
\Phi^+:&(\mathbb R^2,0) \to (\mathbb R^4,a^+)\\
            &(y,z)\mapsto a^+ +(y, z, \phi(y,z), \psi(y,z)),
\end{split}
\end{equation}
$j^1\phi (0,0)=j^1\psi (0,0)=0.$
\begin{equation}
\begin{split}
\Phi^-:& (\mathbb R^2,0) \to (\mathbb R^4,a^-)\\
 &(u,v)\mapsto a^- +(u,\xi (u,v), \zeta (u,v),v),
\end{split}
\end{equation}
$j^1\xi (0,0)=j^1\zeta (0,0)=0.$
In these coordinates, the local expression of the map $\Psi_{\lambda}|_{M\times M}$ is given by
\begin{eqnarray}  \Psi_{\lambda}|_{M\times M} &:& (\mathbb R^2,0)\times (\mathbb R^2,0) \to  (\mathbb R^4,0) \nonumber \\
& &((y,z),(u,v)) \mapsto  (\lambda y+ (1-\lambda)u,  \lambda z+ (1-\lambda)\xi(u,v), \nonumber  \\ & & \quad\quad\quad \lambda \phi(y,z)+ (1-\lambda)\zeta(u,v), \lambda \psi(y,z)+ (1-\lambda)v) \nonumber \end{eqnarray}
where,  to simplify, we have assumed $\lambda a^+ +(1-\lambda)a^{-}=0$, for fixed $\lambda$.

In order to construct the contact map, we first reflect $(M^-,a^-)$ through the point
$0$ to get ${\mathcal R}_0^{\lambda }(M^-),$ parametrized as
\begin{equation*}
{\mathcal R}_0^{\lambda }(\Phi ^-)(u,v)=a^+-({(1-\lambda )\over \lambda }u, {(1-\lambda )\over \lambda }\xi(u,v), {(1-\lambda )\over \lambda }\zeta (u,v),{(1-\lambda )\over \lambda }v).
\end{equation*}
The contact map ${\mathcal K^{\lambda}} :(\mathbb R^2,0)\to (\mathbb R^2,0)$ is then given by
\begin{equation}\label{eq:contactmap}
\begin{split}
\mathcal K^{\lambda} (y,z)=&(z+{1-\lambda \over \lambda }\xi ({-\lambda \over 1-\lambda }y,{-\lambda \over 1-\lambda }\psi (y,z)),\\
&\phi (y,z)+{1-\lambda \over \lambda }\zeta ({-\lambda \over 1-\lambda }y,{-\lambda \over 1-\lambda }\psi (y,z)).
\end{split}
\end{equation}

The following theorem distinguishes the $A_{\mu}, \,  1\leq \mu \leq 4$ singularities  of equidistants, in terms of the bi-local geometry of $M$.

\begin{thm}\label{prop-A_k} Let $a^+\in M^+$, $a^-\in M^-$, so that $\lambda a^+ + (1-\lambda)a^-=0$ is a singular point of $\Psi_{\lambda}|_{M\times M}$.
For a pair of vectors $({\bf u},{\bf v})$ in $\mathbb R^4$, such that  ${\bf u}$ is in the direction of $1$-parallelism of $(a^+,a^-)$ and ${\bf v}\in N_{a^+}M^+\cap N_{a^-}M^-$ is in the common normal direction, let $\eta_{+}$ and $\eta_{-}$ be the normal curvature of $M^{+}$ and $\mathcal R_0^{\lambda}(M^-)$ along {\bf v} in the common direction ${\bf u}$. Then $0$ is a singular point of $\Psi_{\lambda}|_{M\times M}$ of type $A_k$ if and only if
%\begin{equation}
%\label{eq:curvatures}
%\, \, \,  \frac{\partial^j \eta_{+}^{(j)}}{\partial y^j}(0)=(-1)^j\left( \frac{\lambda}{1-\lambda}\right)^{j-1} \frac{\partial^j %\eta_{-}^{(j)}}{\partial y^j}(0),\, j=0, \ldots, k-1,
%\end{equation} where $\eta_{+}^{(j)}$ and $\eta_{-}^{(j)}$ denotes the $j$-order derivatives of $\eta_{+}$ and $\eta_{-}$ %respectively.
\begin{eqnarray}
\label{eq:curvatures}
\quad\quad \eta_{+}^{(j)}(0)&=&(-1)^{j+1}\left( \frac{\lambda}{1-\lambda}\right)^{j+1} \eta_{-}^{(j)}(0) \ ,\, j=0, \ldots, k-1, \\
\eta_{+}^{(k)}(0)&\neq&(-1)^{k+1}\left( \frac{\lambda}{1-\lambda}\right)^{k+1} \eta_{-}^{(k)}(0) \ ,\, \label{other}
\end{eqnarray} where $\eta_{+}^{(j)}$ and $\eta_{-}^{(j)}$ denote the $j$-order derivatives of $\eta_{+}$ and $\eta_{-}$ respectively.

%(i) The pair $(u,v)$ is a contact pair of $M^+$ and of $\mathcal R_0^{\lambda}(M^-)$at $a^+$, or
%The pair $(u,v)$ is not a contact pair of either $M^+$ or $\mathcal R_0^{\lambda}(M^-)$ at $a^+$, but their normal curvature %along $u$ in the $v$ direction are in %proportion $\frac{\lambda}{1-\lambda}$ at $a^+$  \big(or equivalently, $(u,v)$ is not a %contact pair of either $M^-$ or $\mathcal R_0^{1-\lambda}(M^+)$ at $a^-$, but %their normal curvature along $u$ in the $v$ %direction are in proportion $\frac{1-\lambda}{\lambda}$ at $a^-$ \big).
\end{thm}

\begin{proof}
We can solve the first equation $\mathcal K^{\lambda}_1=0$ in (\ref{eq:contactmap}), as $z=z(y)$,
so that the contact map $\mathcal K^{\lambda}$ is $\mathcal K$-equivalent to the suspension of
\begin{equation}
\begin{split}
\theta_{\lambda} : &\mathbb R\to \mathbb R\\
&y\mapsto \phi(y,z(y))+{1-\lambda \over \lambda }\zeta({-\lambda \over 1-\lambda }y,{-\lambda \over 1-\lambda }\psi (y,z(y))).
\end{split}
\end{equation}
The point $0$ is a singularity of type $A_k$ of $\theta_{\lambda} $ if and only if
\begin{eqnarray}
\label{eq:derivatives}
\frac{\partial^j \phi_{}}{\partial y^j}(0)&=&(-1)^{j-1}\left( \frac{\lambda}{1-\lambda}\right)^{j-1} \frac{\partial^{j} \zeta_{}}{\partial y^j}(0),\, j=1, \ldots, k, \\
\frac{\partial^j \phi_{}}{\partial y^j}(0)&\neq&(-1)^{j-1}\left( \frac{\lambda}{1-\lambda}\right)^{j-1} \frac{\partial^{j} \zeta_{}}{\partial y^j}(0),\, j=k+1, \label{other2}
\end{eqnarray}
noting that condition (\ref{eq:derivatives}) for $j=1$ is the condition of $1$-parallelism.

Letting $\alpha_{+}$ and $\alpha_{-}$ be curves in $M^+$ and ${\mathcal R}_0^{\lambda }(M^-)$ given by
\begin{equation*}
\begin{split}
\alpha _+(y)=&(y,z(y),\phi (y,z(y),\psi (y,z(y))\\
\alpha _- (y)=&(y, {\lambda -1\over \lambda} \xi ({\lambda \over \lambda -1}y,{\lambda \over \lambda -1}\psi(y,z(y)),\\
&{\lambda -1\over \lambda}\zeta ({\lambda \over \lambda -1}y,{\lambda \over \lambda -1}\psi(y,z(y))), \psi (y,z(y)))
\end{split}
\end{equation*}
and letting $\eta _{+}(y)$ and $\eta _{-}(y)$ be the projections of the
normal curvatures of $\alpha_{+}$ and $\alpha_{-}$ in the common
normal direction {\bf v}, then
$$
\eta _{+}(y)={\partial^2\phi \over \partial y^2}(y,z(y)) \  {\rm and} \  \eta _{-}(y)={\partial^2\zeta \over \partial y^2}(y,z(y)).
$$
So, equations (\ref{eq:derivatives})-(\ref{other2}) reduce to equations (\ref{eq:curvatures})-(\ref{other}).
\end{proof}

We now look at the bi-local description of $2$-parallel pairs.

Suppose $(a^+, a^-)$ is a pair of 2-parallel points. Then, we can choose coordinates in a neighbourhood of  $a^+$ and $a^-$ as follows:
\begin{equation}
\begin{split}
\Phi^+:& (\mathbb R^2,0) \to (\mathbb R^4,a^+)\\
 & (y,z)\mapsto a^+ + (y, z, \phi(y,z), \psi(y,z)),
\end{split}
\end{equation}
$j^1\phi (0,0)=j^1\psi (0,0)=0.$
\begin{equation}
\begin{split}
\Phi^-:& (\mathbb R^2,0) \to (\mathbb R^4,a^-)\\
 & (u,v)\mapsto a^- + (u, v,\xi (u,v), \zeta (u,v)),
\end{split}
\end{equation}
$j^1\xi (0,0)=j^1\zeta (0,0)=0.$

Again, for simplicity we assume that for $\lambda$ fixed, $\lambda a^+ +(1-\lambda)a^{-}=0$.
Now the contact map ${\mathcal K^{\lambda}} :(\mathbb R^2,0)\to (\mathbb R^2,0)$ is
\begin{equation}
\begin{split}
\label{eq:contactmap2}
\mathcal K^{\lambda} (y,z)=&(\phi (y,z)+{1-\lambda \over \lambda }\xi ({-\lambda \over 1-\lambda }y,{-\lambda \over 1-\lambda }z),\\
&\psi (y,z)+{1-\lambda \over \lambda }\zeta ({-\lambda \over 1-\lambda }y,{-\lambda \over 1-\lambda }z)
\end{split}
\end{equation}
Let the contact surface $\mathcal C^{\lambda}\subset\mathbb R^4$ be the graph of the contact map $\mathcal K^{\lambda}.$

If $0\in \mathcal C^{\lambda}\subset\mathbb R^4$ is a singular point of type $C^+_{2,2}$ of the contact map $\mathcal K^{\lambda}$, then $\Delta_{\mathcal C^{\lambda}}(0)<0$ \cite{MRR}. It follows that $\mathcal C^{\lambda}$ has two contact pairs at $0$. For each of these, we have the following:

\begin{thm}\label{c+} Let $a^+\in M^+$, $a^-\in M^-$, so that $\lambda a^+ + (1-\lambda)a^- =0\in \mathcal C^{\lambda}\subset\mathbb R^4$ is a singular point of $\mathcal K^{\lambda}$ of type $C^+_{2,2}$. The pair $({\bf u}, {\bf v})$ is a contact pair of $\mathcal C^{\lambda}$ at $0$ if and only if one of the following holds.

(i) The pair $({\bf u}, {\bf v})$ is a contact pair of $M^+$ and of $\mathcal R_0^{\lambda}(M^-)$ at $a^+$ \big(equivalently, $({\bf u}, {\bf v})$ is a contact pair of $M^-$ and of $\mathcal R_0^{1-\lambda}(M^+)$ at $a^-$\big).

(ii) The pair $({\bf u}, {\bf v})$ is not a contact pair of either $M^+$ or $\mathcal R_0^{\lambda}(M^-)$ at $a^+$, but the normal curvatures of $M^+$ and $\mathcal R_0^{\lambda}(M^-)$ along $\bf u$ in the direction of $\bf v$ are in  proportion $\frac{\lambda}{1-\lambda}$ at $a^+$ \big(equivalently, $({\bf u}, {\bf v})$ is not a contact pair of either $M^-$ or $\mathcal R_0^{1-\lambda}(M^+)$ at $a^-$, but the normal curvatures of $M^-$ and $\mathcal R_0^{1-\lambda}(M^+)$ along $\bf u$ in the direction of $\bf v$ have the proportion $\frac{1-\lambda}{\lambda}$ at $a^-$\big).
\end{thm}

\begin{proof}
Let $({\bf u}, {\bf v})$ be a contact pair of the contact
surface $\mathcal C^{\lambda}.$ Without loss of generality we can take ${\bf u}=(1,0,0,0)$ and ${\bf v}=(0,0,1,0).$ Then, since ${\bf v}$ is a binormal direction, it follows that the hessian of the function germ $$\mathcal K^{\lambda}_2(y,z)=\psi (y,z)+{1-\lambda \over \lambda }\zeta ({-\lambda \over 1-\lambda }y,{-\lambda \over 1-\lambda }z)$$ is degenerate and ${\bf u}$ is its kernel.
Then $\frac{\partial^2 \mathcal K^{\lambda}_2}{\partial y^2}(0)=0,$ hence
$$\frac{\partial^2 \psi}{\partial y^2}(0)=-{\lambda \over 1-\lambda } \frac{\partial^2 \zeta}{\partial y^2}(0).$$ As in the proof of Theorem \ref{prop-A_k}, either $\frac{\partial^2 \psi}{\partial y^2}(0)=0$ and $\frac{\partial^2 \zeta}{\partial y^2}(0)=0$ or they are not zero, but the normal curvatures of $M^+$ and $\mathcal R_{0}^{\lambda}(M^-)$ along {\bf v} in the direction of {\bf u} are proportional. Similar statement holds for $M^-$ and $\mathcal R_{0}^{1-\lambda}(M^+).$
\end{proof}

%This can be re-expressed as follows:
%\begin{cor}
%\end{cor}

If $0\in \mathcal C^{\lambda}\subset\mathbb R^4$ is a singular point of type $C^-_{2,2}$, then $\Delta_{\mathcal C^{\lambda}}(0)>0$ \cite{MRR}. It follows that $\mathcal C^{\lambda}$ has no contact pairs at $0$. We thus have:

\begin{thm}\label{c-} Let $a^+\in M^+$, $a^-\in M^-$, so that $\lambda a^+ + (1-\lambda)a^- =0\in \mathcal C^{\lambda}\subset\mathbb R^4$ is a singular point of type $C^-_{2,2}$. Although $a^+\in M^+$ and $a^-\in M^-$ are strongly parallel points, both of the following holds true.

(i) ${M^+}$ and $\mathcal R_0^{\lambda}(M^-)$ do not have any common contact pair at $a^+$ \big(or equivalently,  ${M^-}$ and $\mathcal R_0^{1-\lambda}(M^+)$ do not have any common contact pair at $a^-$\big).

(ii) There is no pair $({\bf u}, {\bf v})\in\mathbb R^4$ with ${\bf u}\in T_{a^+}M^{+}$ and  ${\bf v}\in N_{a^+}M^{+}$, such that the normal curvature along $\bf u$ in the $\bf v$ direction of  ${M^+}$ and of $\mathcal R_0^{\lambda}(M^-)$ are in proportion $\frac{\lambda}{1-\lambda}$ at $a^+$ \big(or equivalently,  the normal curvature along $\bf u$ in the $\bf v$ direction of  ${M^-}$ and of $\mathcal R_0^{1-\lambda}(M^+)$ are in proportion $\frac{1-\lambda}{\lambda}$ at $a^-$\big).

\end{thm}

%\begin{proof}

%\end{proof}

\begin{rem}\label{noparabolic} Generically, $\Delta_{\mathcal C^{\lambda}}\neq 0$ because singular points of  $\mathcal C^{\lambda}\subset\mathbb R^4$ of type $C_{2,3}$ are not unfolded to a stable point of $\Psi_{\lambda}$ (\cite{DRR}). \end{rem}

\section{Geometry of the set of weakly parallel points}

We now extend our geometric investigations in order to describe the set of weakly parallel points of $M$, as this set is naturally related to the set of affine equidistants of $M$ and its singularities.

\subsection{Grassmannian investigation of weakly parallel points}
We start by using the Grassmannian $Gr(2,4)$, the space of  $2$-planes in $\mathbb R^4$.

First,
we  recall the Pl\"ucker coordinates for $Gr(2,4)$.
Let $\bf e_1,e_2,e_3,e_4$ be any basis for $\mathbb R^4$ (not necessarily orthonormal or orthogonal,  no metric is needed or assumed here).  Then, $\bf e_1\wedge e_2, e_1\wedge e_3, e_1\wedge e_4, e_2\wedge e_3, e_2\wedge e_4, e_3\wedge e_4$ is a basis for $\Lambda^2\mathbb R^4$ and we denote by $(p_{12},p_{13},p_{14},p_{23},p_{24},p_{34})$ the coordinates of an element $\pi\in \Lambda^2\mathbb R^4$ in the above basis.

If the bi-vector $\pi\in  \Lambda^2\mathbb R^4$ with coordinates $(p_{12},p_{13},p_{14},p_{23},p_{24},p_{34})$ represents  an element in $Gr(2,4)$, then the bi-vector $\pi'\in \Lambda^2\mathbb R^4$ with coordinates $(kp_{12},kp_{13},kp_{14},kp_{23},kp_{24},kp_{34})$, $0\neq k\in\mathbb R$, represents the same element in $Gr(2,4)$. Thus, defining the equivalence class $[\pi]=\{\pi' \in \Lambda^2\mathbb R^4 \ | \ \pi'=k\pi, k\in\mathbb R^*\}$, it follows that
$[\pi]\in  \mathbb P( \Lambda^2\mathbb R^4)$ has homogeneous coordinates $[p_{12},p_{13},p_{14},p_{23},p_{24},p_{34}]$.

However, not every element $[\pi]\in\mathbb P( \Lambda^2\mathbb R^4)$ lies in $Gr(2,4)$. $\pi$ is in $Gr(2,4)$ iff $\pi$ is an elementary bi-vector, i.e.  $\pi=\bf u\wedge v$, for some ${\bf u,v}\in\mathbb R^4$. Thus $[\pi]\in Gr(2,4)$ iff
$$\pi\wedge\pi=0.$$
In terms of the above coordinates, this translates into the equation
\begin{equation}\label{Plucker}
p_{12}p_{34} + p_{23}p_{14} - p_{13}p_{24}=0.
\end{equation}
The homogeneous coordinates $[p_{12},p_{13},p_{14},p_{23},p_{24},p_{34}]$ subject to constraint (\ref{Plucker}) are  the Pl\"ucker coordinates of $[\pi]\in Gr(2,4)$ with respect to the basis $\bf e_1,e_2,e_3,e_4$ of $\mathbb R^4$. It follows that $dim_{\mathbb R}(Gr(2,4))=4$.

%We are now ready to look at the Grassmannian description of weakly parallel pairs on $M\subset \mathbb R^4$.
Now, consider the Gauss map $$G :  M\to Gr(2,4) \ , \ \mathbb R^4\supset M\ni a\mapsto [T_aM]\in Gr(2,4).$$

The Gauss map fails to be injective precisely for (non-diagonal) strongly parallel pairs, i.e, $a_1\neq a_2\in M$, such that $G(a_1)=G(a_2)$. Thus, for a residual set of embeddings $M\subset\mathbb R^4$, $G:M\to Gr(2,4)$ is an immersion with transversal double points and such a $[\pi]\in G(M)$ whose neighborhood in $G(M)$ is not homeomorphic to $\mathbb R^2$ is the common tangent plane for a (non-diagonal) $2$-parallel pair $(a_1, a_2)\in M\times M$.

Consider also the double Gauss map:
$$G\times G: M\times M\to  Gr(2,4)\times Gr(2,4) \ , \
(a_1,a_2)\mapsto ([\pi_1], [\pi_2])$$
Then, $[\pi_1]$ and $[\pi_2]$ are weakly parallel, iff
\begin{equation}\label{weak} \pi_1\wedge\pi_2=0.\end{equation}
And we denote
$$W= \{([\pi_1],[\pi_2]) \in Gr(2,4)\times Gr(2,4) | \ \pi_1\wedge\pi_2=0\}.$$

In terms of the Pl\"ucker coordinates for  $Gr(2,4)$,
\begin{equation}\label{pi1} [\pi_1]= [p_{12},p_{13},p_{14},p_{23},p_{24},p_{34}] \ , \ p_{12}p_{34} + p_{23}p_{14} - p_{13}p_{24}=0 \ ,\end{equation}
\begin{equation}\label{pi2} [\pi_2]= [q_{12},q_{13},q_{14},q_{23},q_{24},q_{34}] \ , \ q_{12}q_{34} + q_{23}q_{14} - q_{13}q_{24}=0  \ , \end{equation}
condition (\ref{weak}) translates into
\begin{equation}\label{weakPlucker}
p_{12}q_{34} + p_{34}q_{12} + p_{14}q_{23} + p_{23}q_{14} - p_{13}q_{24} - p_{24}q_{13} = 0. \end{equation}
Thus, equations (\ref{pi1}), (\ref{pi2}) and (\ref{weakPlucker}) define coordinates for an element $([\pi_1],[\pi_2])$ of the $7$-dimensional subvariety $W\subset(Gr(2,4)\times Gr(2,4))$.

We denote by $W_{reg}$ the set of smooth points of $W$, and by $Sing(W)$ the set of singular points of $W$.

\begin{prop}\label{singW}
Away from the diagonal, $W$ is a smooth hypersurface  of $Gr(2,4)\times Gr(2,4)$.
\end{prop}
\begin{proof}
First, note that each of the equations (\ref{pi1}) and (\ref{pi2}) define smooth submanifolds $Gr(2,4)\subset \mathbb P(\Lambda^2\mathbb R^4)$ and, similarly, equation (\ref{weakPlucker}) defines a smooth submanifold of $\mathbb P(\Lambda^2\mathbb R^4)\times\mathbb P(\Lambda^2\mathbb R^4)$. Thus, $W$ is singular only where these three  submanifolds of
$\mathbb P(\Lambda^2\mathbb R^4)\times\mathbb P(\Lambda^2\mathbb R^4)$ do not intersect transversaly. By straightforward computation, we see that the rank of the matrix of the derivatives of equations  (\ref{pi1}), (\ref{pi2}) and (\ref{weakPlucker}) is not maximal iff $\forall \ 1\leq i<j\leq 4  , \ p_{ij}/q_{ij}=k\in\mathbb R^*$. It follows that $Sing(W)=\{([\pi_1],[\pi_2])\in Gr(2,4)\times Gr(2,4) \ | \ [\pi_1]=[\pi_2]\}$.
\end{proof}

Now, as $Gr(2,4)\times Gr(2,4)$ fibers (trivially) over $Gr(2,4)$, say, via the first projection $Pr_1$, this induces a sub-bundle $W \to Gr(2,4)$, $([\pi_1],[\pi_2])\mapsto [\pi_1]$, which may not be trivial. Its typical fiber $W_{[\pi_1]}$  is a $3$-variety, which can locally be described as follows.

Chose a basis $\bf e_1,e_2,e_3,e_4$ for $\mathbb R^4$ such that $[\pi_1]=[\bf e_1\wedge e_2]$. Then, $[\pi_1]=[1,0,0,0,0,0]$, and $[\pi_2]=[q_{12},q_{13},q_{14},q_{23},q_{24},q_{34}]\in W_{[\pi_1]}$ iff $q_{12}q_{34} + q_{23}q_{14} - q_{13}q_{24}=0$ and $q_{34}=0$, that is,
$$[\pi_2]\in W_{[\pi_1]}\iff [\pi_2]=[q_{12},q_{13},q_{14},q_{23},q_{24},0] \ , \  q_{23}q_{14} - q_{13}q_{24}=0 \ ,$$ or equivalently,
\begin{equation}\label{fiber1}
[\pi_2]\in W_{[\pi_1]}\iff [\pi_2]=[1,\alpha,\beta,\gamma,\delta,0] \ , \  \beta\gamma - \alpha\delta=0 \ ,\end{equation}
in other words, close to $\alpha=\beta=\gamma=\delta=0$,
\begin{equation}\label{fiber2}
W_{[\pi_1]}=\{(\alpha,\beta,\gamma,\delta)\in\mathbb R^4 \ | \ \alpha\delta-\beta\gamma=0 \} \ . \end{equation}
Thus, we have a refinement of Proposition \ref{singW}, that is,
\begin{prop}\label{Wpi2}
In a neighborhood of $[\pi_2]=[\pi_1]$,  the $3$-variety $W_{[\pi_1]}$ is a cone. \end{prop}

The following theorem, which follows from standard transversality arguments, describes how  affine equidistants $E_{\lambda}(M)$ are related to the intersection of
$W$ and $G(M)\times G(M)$.

\begin{thm}\label{thmgrass} Let $M\subset \mathbb R^4$ be a generic embedding and $(a,b)$ be a weakly parallel pair on $M$.

\

\noindent (i)  Let $(a,b)$ be a $1$-parallel pair, so that $(G(a),G(b))\in W_{reg}$. If $\Psi_{\lambda}|_{M\times M}: (\mathbb R^2\times\mathbb R^2, (a,b)) \to (\mathbb R^4,\lambda a+(1-\lambda)b)$ has a stable singularity (of type $A_k$, $k=1,2,3,4$), then $G(M)\times G(M)$ is transverse to $W_{reg}$ at $(G(a),G(b))$.

\

\noindent (ii)   Let $(a,b)$ be a $2$-parallel pair, so that  $(G(a),G(b))\in  Sing(W)$. If $\Psi_{\lambda}|_{M\times M}: (\mathbb R^2\times\mathbb R^2, (a,b)) \to (\mathbb R^4,\lambda a+(1-\lambda)b)$ has a stable singularity (of type $C^{\pm}_{2,2}$) then $(a,b)$ is a transversal  double point of the Gauss map.

\end{thm}

\subsection{Geometric description of the set of weakly parallel points}

We emphasize that, from Theorem \ref{thmgrass}, for generic embeddings of smooth closed surfaces in $\mathbb R^4$ there are only double points of Gauss map. There are no triple, quadruple... points of the Gauss map, generically.

Therefore we obtain the following corollary of Theorem \ref{thmgrass}:

\begin{cor}\label{corgrass} For generic embeddings of smooth closed surfaces in $\mathbb R^4$, strongly parallel (nonidentical) points come only in pairs and there are only finite numbers of such pairs.
\end{cor}

An interesting question, whose answer is unknown to us, is whether there exists any embedded compact surface $M\subset \mathbb R^4$ without nonidentical $2$-parallel points, in other words, such that the Gauss map $G: M\to Gr(2,4)$ is injective.

%\subsection{Local geometry of the set of weakly parallel points}

\begin{notation} For  $p\in M$, let ${\mathcal W}_p\subset M$ denote the set of weakly parallel points to $p$ and let ${\mathcal W}_p^{q}$ denote the germ of  ${\mathcal W}_p$ at $q\in M$. \end{notation}

\begin{rem} It is easy to see that $G(\mathcal W_p)\subset W$ where the latter is described in Propositions \ref{singW} and \ref{Wpi2}. \end{rem}

Then, the following theorem describes ${\mathcal W}_p^{q}$ in all possible situations.

\begin{thm}\label{weaklyset}  For a generic embedding of $M$ into $\mathbb R^4$, cf. Definition \ref{genericembedding} and Theorem \ref{eq:stable},  the following hold.

\

\noindent (1) If $q$ is $1$-parallel to $p$, then  ${\mathcal W}_p^{q}$ is a germ of smooth curve.

\

\noindent (2) If $q$ is $2$-parallel to $p$, then:

\

\noindent (i) If $q$ is an elliptic point of $M$, then ${\mathcal W}_p^{q}=\{q\}$.

\

\noindent (ii) If $q$ is a parabolic point of $M$, then ${\mathcal W}_p^{q}$ is a  singular curve  with a cusp singularity at $q$ which is tangent  to the asymptotic direction at $q$ (this is generic for $q=p$, as a generic embedding has a parabolic point, or in a $1$-parameter family of embeddings for $q\neq p$, cf. Remark \ref{noparabolic}).

\

\noindent (iii) If $q$ is a hyperbolic point of $M$, then ${\mathcal W}_p^{q}$ is a singular curve with a transversal double point at $q$ so that each branch of ${\mathcal W}_p^{q}$ is a smooth curve
tangent to an asymptotic direction at $q$.

\end{thm}

\begin{proof} %We present two proofs. The first one is a straightforward computation:
If the points $p, q \in M$ are $1$-parallel then the germs of $M$ at $p=(p_1,p_2,p_3,p_4)$ and at $q=(q_1,q_2,q_3,q_4)$ can be parametrized in the following way
$F(x,y)=(p_1+x,p_2+y,p_3+f_3(x,y),p_4+f_4(x,y))$ and $G(u,v)=(q_1+u,q_2+g_2(u,v),q_3+g_3(u,v),q_4+v)$ respectively, where $f_3, f_4, g_2, g_4$ are smooth function-germs vanishing at $(0,0)$ such that $df_3|_{(0,0)}=df_4|_{(0,0)}=dg_2|_{(0,0)}=dg_3|_{(0,0)}=0$. The point $G(u,v)$ is weakly parallel to $p$ if the Jacobian of the map
\begin{equation}\label{FG}
(x,y,u,v)\mapsto \lambda F(x,y)+(1-\lambda) G(u,v)
\end{equation}
vanishes at the point $(0,0,u,v)$. The Jacobian of the map (\ref{FG}) at $(0,0,u,v)$ has the  form $\frac{\partial g_3}{\partial u}(u,v)$. Generically $d(\frac{\partial g_3}{\partial u})|_{(0,0)}\ne 0,$ therefore ${\mathcal W}_p^q$ is a germ at $q$ of a smooth curve.

If  the points $p, q \in M$ are $2$-parallel then the germs of $M$ at $p=(p_1,p_2,p_3,p_4)$ and at $q=(q_1,q_2,q_3,q_4)$ can be parametrized in the following way
$F(x,y)=(p_1+x,p_2+y,p_3+f_3(x,y),p_4+f_4(x,y))$ and $G(u,v)=(q_1+u,q_2+v,q_3+g_3(u,v),q_4+g_4(u,v))$ respectively, where $f_3, f_4, g_3, g_4$ are smooth function-germs vanishing at $(0,0)$ such that $df_3|_{(0,0)}=df_4|_{(0,0)}=dg_3|_{(0,0)}=dg_4|_{(0,0)}=0$.

The point $G(u,v)$ is weakly parallel to $p$ if the Jacobian of the map (\ref{FG}) vanishes at $(0,0,u,v)$. It is easy to see that the Jacobian of the map (\ref{FG}) at $(0,0,u,v)$ is $Jac(g_3,g_4)(u,v)$, i.e. the Jacobian of the map $(g_3,g_4)$ at $(u,v)$. It is also easy to see $d(Jac(g_3,g_4))|_{(0,0)}$ vanishes.

The Hessian of the function $(u,v)\mapsto Jac(g_3,g_4)(u,v)$ at $(0,0)$ is equal to $4\Delta_M(q)$. Therefore if  $q$ is an elliptic point, then ${\mathcal W}_p^q=\{q\}$, if $q$ is a parabolic point, then ${\mathcal W}_p^q$ is a singular curve with a cusp singularity at $q$ which is tangent  to the asymptotic direction at $q$, and finally if $q$ is a hyperbolic point, then ${\mathcal W}_p^q$ consists of the crossing of two smooth curves at $q$, each one tangent to an asymptotic direction at $q$.

We can also interpret the above calculations in terms of singularities of projections into planes. In fact,
%For the second proof, first
 let $\rho_{p}: M \to N_pM$ be the projection of $M$ into the $2$-plane $N_pM= \mathbb R^2,$ which is fixed.

Then  the singular set of the projection,
 $$\Sigma \rho_p=\{q \in M|\,  \, \text{there exists some}\, {\bf v}\in T_qM, {\bf v} \in \, \text{ker} \rho_p\}$$ coincides with the set
${\mathcal W}_p.$ Given $q \in {\mathcal W}_p,$ we use  the above local parametrizations  to study ${\mathcal W}_{p}^q.$

%We let $\phi^+(y,z)$ be the local coordinates in a neighbourhood of $p$ and $\phi^{-}(u,v)$ the coordinates around $q.$

%Then, $\rho_p: (M^{-}, q) \to \mathbb R^2$ is given by $\rho_p(u,v)=(\zeta(u,v), v).$
If points $p, q \in M$ are $1$-parallel then the germs of $ M $ at $p=(p_1,p_2,p_3,p_4)$ and at $q=(q_1,q_2,q_3,q_4)$ can be parametrized  respectively by
$F(x,y)=(p_1+x,p_2+y,p_3+f_3(x
,y),p_4+f_4(x,y))$ and $G(u,v)=(q_1+u,q_2+g_2(u,v),q_3+g_3(u,v),q_4+v).$ The normal plane of $M$ at $p$ is the plane
$[(0,0,1,0), (0,0,0,1)].$ Hence,  the germ at $q$ of the projection $\rho_p: M \to  N_pM$ is given by
$$\rho_p\circ G(u,v)=(q_3+g_3(u,v), q_4+ v).$$
Thus, as above, ${\mathcal W}_p^q$ is smooth in a neighbourhood of $q$  if and only if $({g_3}_{uu}(0), {g_3}_{uv}(0)) \neq (0,0)$  and this condition is satisfied for generic embeddings of $M$.
On the other hand, if points $p, q \in M$ are $2$-parallel, the germ at $q$ of the projection is given by
$$\rho_p\circ G(u,v)=(q_3+g_3(u,v), q_4+ g_4(u,v)),$$  and we proceed as above.
%For a generic embedding of $M$ into $\mathbb R^4,$ the mapping $\Psi_{\lambda}$ is stable, hence the $1$-jet
%$j^1\Psi_{\lambda}$ is transversal to the Boardman submanifolds $\Sigma^r\subset J^1(M\times M, \mathbb R^4)$ (see %\cite{GG}).
%At $1$ parallel points, if the pair $(a^+, a^-)$ is a fold point of $\Phi_{\lambda}$, then $\zeta_{uu}(0)\neq 0$ and %$\phi_{yy(0)\neq 0.$ If the pair is not a fold, then the $j^1\Psi_{\lambda} \pitchfork \Sigma^r$ at the point $(a^+,a^-)$ %implies
%that $\zeta_{uv}(0)\neq 0$ and $\phi_{yz}(0)\neq 0.$
%\
\end{proof}

Because the case $(2-ii)$ above for $q\neq p$ is only generic in a $1$-parameter family of embeddings, according to Definition \ref{genericembedding} and Theorem \ref{eq:stable},  we now analyze  its bifurcation set.

\begin{prop}\label{pwf} Let $I\ni t$ be an open interval containing $0$, with $M_t$ a generic $1$-parameter family of smooth-surface embeddings  in $\mathbb R^4$ such that  the points $p_t$ and $q_t$ in $M_t$ are strongly parallel $\forall t\in I$ and $q_0$ is a parabolic point of $M_0\subset\mathbb R^4$.
Let $\mathcal W_p^q(t)\subset M_t$ denote the germ of weakly parallel points to $p_t$ at $q_t$. Then, $\mathcal W_p^q(t)$ is described by the Whitney umbrella
\begin{equation}\label{wu} 2u^2-3v^3-2tv^2 = 0 \ , \end{equation}
such that a curve $\mathcal C_{t=t_0}$ on this surface in $\mathbb R^3$ has a smooth branch and an isolated point if $t_0<0$, or is a cusp if $t_0=0$, or  is a looped curve with a  transversal self-crossing if $t_0>0$. These three cases correspond to the point $q_{t_0}=(u,v)=(0,0)$ being an isolated point for $t_0<0$ (elliptic case),
 a cuspidal point for $t_0=0$ (parabolic case),  or a
transversal double point for $t_0>0$ (hyperbolic case).
\end{prop}

\begin{proof} Following the same notation of the proof of Theorem \ref{weaklyset}, with $t$ denoting the parameter of the family of embeddings and assuming $q_t$ is parabolic for $t=0$, the germ of $M_t$ at $q_t$ can be put after translation to the form\footnote{In general, the tangent plane to $M_t$ at $q_t$ will change with $t$, but we can adopt an orthonormal moving frame such that $T_{q_t}M_t=span<(1,0,0,0), (0,1,0,0)>, \forall t$.} $g_t(u,v)=(u,v, g^3_t(u,v), g^4_t(u,v))$, where $g^3_t(u,v)=  u^2+v^3+tv^2 +V_t(u,v)$ and $g^4_t(u,v)= uv + W_t(u,v)$, with $V_t$ and $W_t$  of third or higher order in $(u,v)$ for all $t$. 

The point $p_t$ being $2$-parallel to $q_t$, the germ of $M_t$ at $p_t$ is, after translation, of the general form $f_t(x,y)=(x,y,f_t^3(x,y), f_t^4(x,y))$, with $f_t^3$ and $f_t^4$ of second order in $(u,v)$ for all $t$. 

Thus, as before, $g_t(u,v)$ is weakly parallel to $p_t$ if the Jacobian of the map (\ref{FG}) vanishes at $(0,0,u,v)$ and this Jacobian  is the same as the Jacobian of the map $(g_t^3,g_t^4)$ at $(u,v)$, which is of the form $J(u,v,t)=2u^2-3v^3-2tv^2 + R_t(u,v)$, where $R_t$ is of third or higher order in $(u,v)$ for all $t$. We now apply the following lemma: 

\begin{lem} The Jacobian $J(u,v,t)=2u^2-3v^3-2tv^2 + R_t(u,v)$, with $R_t$ of third or higher order in $(u,v)$, $\forall t$, can be put for small $t$, by a smooth near-identity change of coordinates of the form $(u,v,t)\mapsto (U(u,v,t), V(u,v,t), t)$, to the normal form $H(U,V,t)=2U^2-3V^3-2tV^2(1+\phi(U,V,t))$, with $\phi$ a smooth function satisfying $\phi(0,0,t)=0$, for small $t$. \end{lem}   
\begin{proof} 
Start by writing  $R_t(u,v)=t\big(\psi_3(v,t)+u\psi_2(v,t)+ 2u^2\psi_1(u,v,t)\big)$, where $\psi_3$ is of order at least $3$ in $v$, $\forall t$,  $\psi_2$ is of order at least $2$ in $v$, $\forall t$, and $\psi_1(0,0,t)=0$ .  Then, $J(u,v;t)= 2u^2(1+t\psi_1(u,v,t))-3v^3(1-t\widetilde\psi_3(v,t)) -2tv^2(1-u\widetilde\psi_2(v,t))$, where $\widetilde\psi_3(v,t)=\psi_3(v,t)/3v^3$, $\widetilde\psi_2(v,t)=\psi_2(v,t)/2v^2$. Thus, setting $V(u,v,t)=V(v,t)=v\sqrt[3]{1-t\widetilde\psi_3(v,t)}$ and $U(u,v,t)=u\sqrt{1+t\psi_1(u,v,t)}$, we note that $(u,v,t)\mapsto(U,V,t)$ is a near-identity transformation for small $t$, therefore invertible, so that we can write $J(u,v,t)=H(U,V,t)=2U^2-3V^3-2tV^2(1+\phi(U,V,t))$, where $\phi$ is a smooth function satisfying $\phi(0,0,t)=0$, for small $t$. 
\end{proof}

It follows that, for small $t$ and in  a neighborhood of $(U,V)=(0,0)$, the curve $\mathcal C_{t=t_0}'$, which is obtained as the section $\{H(U,V,t=t_0)=0\}$,  is a small deformation of the curve $\mathcal C_{t=t_0}$, which is obtained as the section $\{h(u,v,t=t_0)=0\}$, where $h(u,v,t)=2u^2-3v^3-2tv^2$. In particular, for $t_0=0$ the curve $\mathcal C_{t=t_0}'$ is a cusp, just as $\mathcal C_{t=t_0}$, for  $t_0<0$ the curve $\mathcal C_{t=t_0}'$ has a smooth branch and an isolated point at $(0,0)$, just as $\mathcal C_{t=t_0}$, and for $t_0>0$ the curve $\mathcal C_{t=t_0}'$ is a looped curve with a transversal self-crossing at $(0,0)$, just like $\mathcal C_{t=t_0}$. 
 \end{proof}

\begin{rem}\label{famWp} In the same vein, when the embedding is fixed and $q=p$, if $s\in I$ is a parameter along a curve $p(s)\subset M$ such that  $p(0)$ is a parabolic point of $M\subset\mathbb R^4$, then by slightly adapting the above reasoning  we can easily see that the family of germs $\mathcal W_p^p(s)$ is described by the Whitney umbrella (\ref{wu}), just renaming  $t\mapsto s, \ (u,v)\mapsto (x,y)$. \end{rem}

\begin{rem} As a last remark, we note that two distinct points $q,q'\in \mathcal W_p$ need not be weakly parallel to each other. For instance, if $[\pi_1]=G(p)=[\bf e_1\wedge e_2]$, we may have that $G(q)=[\bf e_1\wedge e_3]$ and $G(q')=[\bf e_2\wedge e_4]$. We also note that, if $(p,q)$ is a strongly parallel pair ($p\neq q$), the local geometry of $p$ and $q$ can be distinct (one elliptic, the other hyperbolic, etc), thus in general ${\mathcal W}_p^{q}$ and ${\mathcal W}_q^{p}$ can be of distinct types.
 \end{rem}

 \subsection{Illustrations}

 We now provide examples of Theorem \ref{weaklyset} and Proposition \ref{pwf}, this latter in the form of Remark \ref{famWp}.

\begin{ex}\label{ex1}
Let us consider the following embedding of a torus into the affine space $\mathbb R^4$(\cite{GMRR}),
$F(x,y)=(f_1(x,y),f_2(x,y),f_3(x,y),f_4(x,y))$, 
$$f_1(x,y)=\cos(x) \left(1-\frac{\cos(y)}{10}\right)+\frac{1}{10} \sin(x) \sin(y),$$
$$f_2(x,y)=\left(1-\frac{\cos(y)}{10}\right) \sin(x)-\frac{1}{10} \cos(x) \sin(y),$$
$$f_3(x,y)=\cos(2x) \left(1-\frac{2 \cos(y)}{5}\right)+\frac{4}{5} \sin(2x) \sin(y),$$
$$f_4(x,y)=\left(1-\frac{2 \cos(y)}{5}\right) \sin(2x)-\frac{4}{5} \cos(2x) \sin(y).$$
The curves of parabolic points on this torus are given by
$$y=\pm 2 \arctan\left(\sqrt{\frac{1}{5} \left(-4+\sqrt{41}\right)}\right).$$

%In Fig. 1 we present on the $x,y$-plane the curve of weakly parallel points to a parabolic point
%$(0,2 \arctan\left(\sqrt{\frac{1}{5} \left(-4+\sqrt{41}\right)}\right)$. At the parabolic point (marked by a black dot) the curve has a cusp singularity.

Fig. 1 presnts the curve of weakly parallel points on the $x,y$-plane to a hyperbolic point $(\pi,\pi)$ (or elliptic point $(0,0)$). All points marked by black dots on Fig. 1 are strongly parallel. Elliptic points $(0,0)$ and $(\pi,0)$ are isolated points of the the curve. There are transversal self-intersections of the curve in hyperbolic points $(0,\pi)$ and $(\pi,\pi)$.
\end{ex}
 %\begin{center}
 %\includegraphics{parabolic4_gr1.eps}

 %{\small {\bf Figure 1.}
 %The set of weakly parallel points to a parabolic point. }
 %\end{center}

 \begin{center}
 \includegraphics{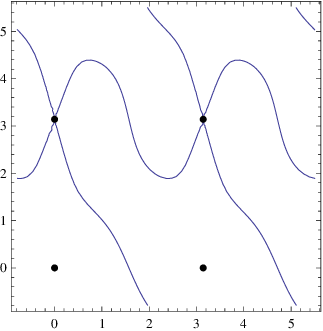}

  {\small {\bf Figure 1.}
  Set of weakly parallel points to an elliptic or hyperbolic point. }
\end{center}

\begin{ex} Let us again consider the torus from Example \ref{ex1}. In Figures  2 to 4 we preset the bifurcation of $\mathcal W_p^p$ - the germ at a point $p$  of the curve of weakly parallel points to $p$ - when we change $p$ from a hyperbolic point to a parabolic point and then to an elliptic point.
For $p$ we chose a point with the following  coordinates on the $(x,y)$-plane: 
$$\left(s,2 \arctan\left(\sqrt{\frac{1}{5} \left(-4+\sqrt{41}\right)}\right)+s\right)$$

 For $s=0$ the point $p$ is parabolic and at this parabolic point (marked by a black dot) the curve has a cusp singularity, cf. Fig. 3, which also shows  the curve of weakly parallel points to this parabolic point. 
 
 For sufficiently small positive $s$ the point $p$ is hyperbolic (cf. Fig. 2) and for sufficiently small negative $s$ the point $p$ is elliptic (cf. Fig 4). 
 The dotted lines on Figs. 2-4 are lines of parabolic points. From the figures we see that the bifurcation of the set $\mathcal W_p^p$ when we change $s$ is diffeomorphic to the Whithney umbrella, which is presented on Fig. 5. 
\end{ex}

\begin{center}
\includegraphics{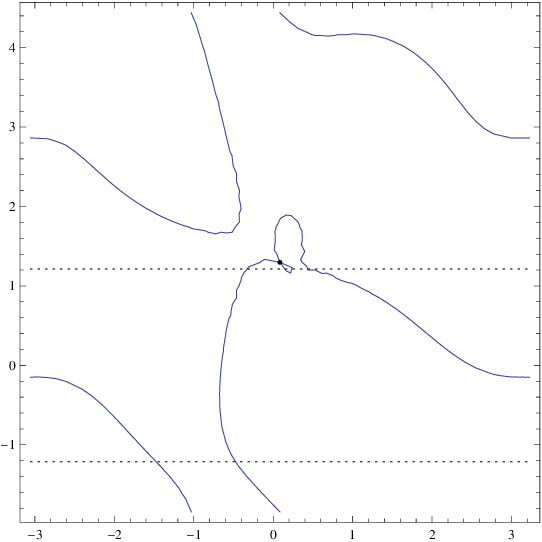}

{\small {\bf Figure 2.} Set of weakly parallel points to a hyperbolic point (s=0.085). }
\end{center}

\begin{center}
\includegraphics{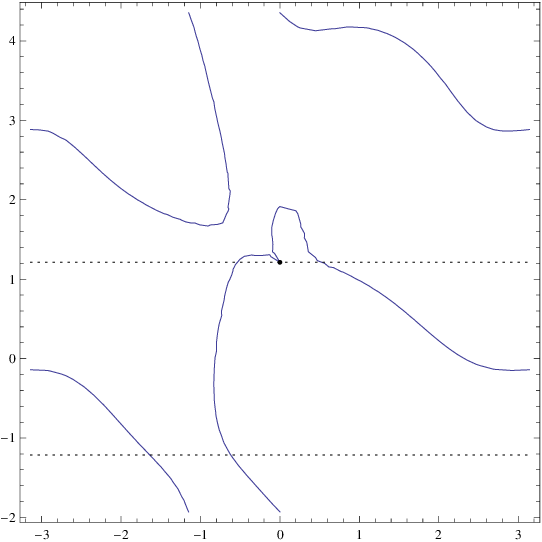}

{\small {\bf Figure 3.} Set of weakly parallel points to a parabolic point (s=0). }
\end{center}

\begin{center}
\includegraphics{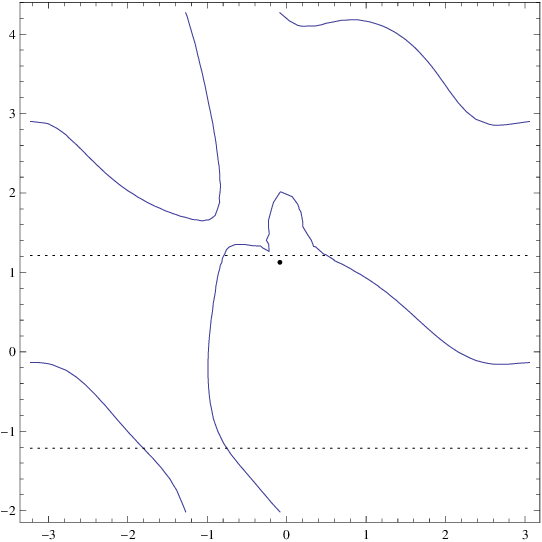}

{\small {\bf Figure 4.} Set of weakly parallel points to an elliptic point $(s=-0.085)$. }
\end{center}

\begin{center}
 \includegraphics{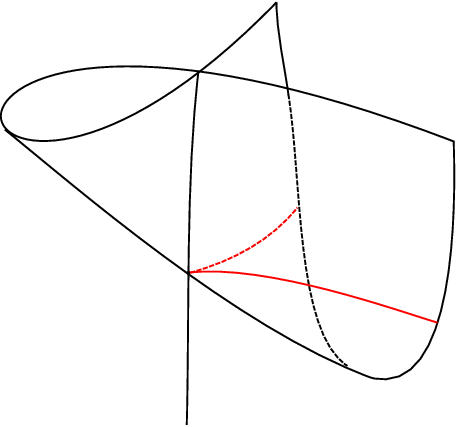}

{\small {\bf Figure 5.} The bifurcation of the germ, at a parabolic point $p$, of the set of weakly parallel points to $p$. }
\end{center}


\begin{thebibliography}{99}


%\bibitem{A-b} V. I. Arnol'd, \emph{
%Critical points of functions on a manifold with boundary, the
%simple Lie groups $B_4$, $C_k$, and $F_4$ and singularities of
%evolutes}, Russ. Math. Surv. 33(1978), 99-116.

%\bibitem{AGV} V. I. Arnol'd, S. M. Gusein-Zade, A. N. Varchenko, \emph{ Singularities of Differentiable Maps}, Vol. 1, Birhauser, Boston, 1985.

\bibitem{Ber} M. V. Berry, \emph{Semi-classical mechanics in phase space: A study of Wigner's function},
Philos. Trans. R. Soc. Lond. A 287 (1977) 237-271.

%\bibitem{CDR} M. Craizer, W. Domitrz, P. de M. Rios \emph{Even-dimensional improper affine spheres}, J. Math. Anal. Appl. 421 (2015) 1803--1826.

%\bibitem{D} W. Domitrz,
%\emph{Local symplectic algebra of quasi-homogeneous curves},
%Fundamenta Mathematicae 204 (2009), 57-86.

%\bibitem{DJZ} W. Domitrz, S. Janeczko, M. Zhitomirskii,
%\emph{Symplectic singularities of varietes: the method of
%algebraic restrictions}, J. reine und angewandte Math. 618 (2008),
%197-235.

%\bibitem{DR} W. Domitrz, J. H. Rieger, \emph{Volume preserving subgroups of
%$\mathcal A$ and $\mathcal K$ and singularities in unimodular
%geometry}, Mathematische Annalen
  % 345(2009), 783�-817.

\bibitem{DMR} W. Domitrz,  M. Manoel, P. de M. Rios, {\it The Wigner caustic on shell and singularities of odd functions}, J. Geometry and Physics 71 (2013) 58-72.

\bibitem{DRs} W. Domitrz, P. de M. Rios, {\it Singularities of equidistants and global centre symmetry sets of Lagrangian submanifolds}, Geom. Dedicata 169 (2014) 361-382.

\bibitem{DRR} W. Domitrz, P. de M. Rios, M. A. S. Ruas, \emph{Singularities of affine equidistants: projections and contacts},  J. Singul. 10 (2014), 67–81.

\bibitem{GMRR} R. A. Garcia, D. Mochida, M. C. Romero Fuster, M. A. S. Ruas,
\emph{Inflection points and topology of surfaces in 4-space},
Trans. Amer. Math. Soc. 352 (2000), no. 7, 3029–3043.

\bibitem{GH} P. J. Giblin, P. A. Holtom, \emph{The centre symmetry set}, Geometry and Topology of Caustics,
Banach Center Publications, Vol 50, Warsaw, 1999, 91-105.

\bibitem{GZ1} P. J. Giblin, V. M. Zakalyukin, \emph{Singularities of systems of chords}, Funct. Anal. Appl.
36 (2002) 220-224.

\bibitem{GZ2} P. J. Giblin, V. M. Zakalyukin, \emph{Singularities of centre symmetry sets}, Proc. London
Math. Soc. 90 (2005) 132-166.

%\bibitem{GZ3} P. J. Giblin, V. M. Zakalyukin,\emph{Recognition of centre symmetry set
%singularities}, Geom. Dedicata 130(2007), 43-58.

%\bibitem{Gib2} P. Giblin, \emph{Affinely invariant symmetry sets}, Geometry and Topology of Caustics -- Caustics '06. Banach Center Publications Vol 82 (2008), 71-84.

\bibitem{GJ} P. Giblin, S. Janeczko, \emph{Geometry of curves and surfaces through the contact map}, Topology Appl. 159 (2012) 466-475.


%\bibitem{Go1} V. V. Goryunov, \emph{Simple projecting maps},
%Sib. Math. J. 25, 50-56 (1984).

%\bibitem{Go2} V. V. Goryunov,
%\emph{Singularities of projections of full intersections}, J. Sov.
%Math. 27, 2785-2811 (1984).

%\bibitem{Go3} V. V. Goryunov, \emph{Singularities of projections},
%Singularity theory, Proceedings of the symposium, Trieste, Italy,
%August 19-September 6, 1991, Singapore: World Scientific. 229-247
%(1995).

\bibitem{Jan} S. Janeczko, \emph{Bifurcations of the center of symmetry}, Geom. Dedicata 60 (1996) 9-16.

\bibitem{li} J. A. Little, \emph{On singularities of submanifolds of higher dimensional Euclidean spaces}, Ann. Mat. Pura Appl. (4) 83 (1969) 261-335.


%\bibitem [IJ1] {IJ1} G. Ishikawa, S. Janeczko, \emph{ Symplectic bifurcations of plane curves and isotropic liftings},
%Q. J. Math. \textbf{54}, No.1 (2003), 73-102.

\bibitem{MRR} D. K. H. Mochida, M. C. Romero Fuster, M. A. S. Ruas, \emph{Geometry of surfaces in $4$-space from a contact viewpoint}, Geom. Dedicata 54 (1995) 323-332.

\bibitem{Mont} J. A. Montaldi, \emph{On contact between submanifolds}, Michigan Math. J. 33 (1986) 195-199.


\bibitem{OH} A. M. Ozorio de Almeida, J. Hannay, \emph{Geometry of Two Dimensional Tori in Phase Space:
Projections, Sections and the Wigner Function}, Annals of Physics
138 (1982) 115-154.

%\bibitem{Poi} H. Poincar\'e, \emph{Les M\'ethodes Nouvelles de la M\'echanique C\'eleste}, vol. 3, Gauthier-Villars, Paris, 1892.

%\bibitem{RO1} P. de M. Rios, A. Ozorio de Almeida, \emph{On the propagation of semiclassical Wigner functions },
%J. Phys. A: Math. Gen. 35 (2002) 2609-2617.

%\bibitem{RO2} P. de M. Rios, A. Ozorio de Almeida, \emph{A variational principle for actions on symmetric symplectic spaces},
%J. Geom. Phys. 51, No. 4, 404-441 (2004).

%\bibitem{Rs} P. de M. Rios, \emph{A semiclassically entangled puzzle}, J. Phys. A: Math. Theor. 40 (2007) F1047-F1052.

%\bibitem{Was} G. Wassermann, \emph{Stability of unfoldings in space and time},
%Acta Math. 135(1975), 57-128 .

%\bibitem{Z} V. M. Zakalyukin, \emph{Reconstructions of fronts and caustics depending on a parameter
%and versality of mappings}, J. Sov. Math. 27(1984), 2713-2735.




\end{thebibliography}
\end{document}